\documentclass[11pt]{article} 
\usepackage{epsfig} 
\usepackage{amssymb,amsmath,amsthm,amscd}
\usepackage{latexsym} 
\usepackage{cancel}
\usepackage{xcolor}
\usepackage{amssymb}

\usepackage[margin=.92in]{geometry}
\usepackage{subcaption}
\usepackage{graphicx}
\usepackage{url}
\usepackage{mathrsfs}
\usepackage{multirow,tabularx}

\usepackage{authblk}

\newtheorem{thm}{Theorem}[section] 
\newtheorem{prop}[thm]{Proposition} 
 
\newtheorem{lem}[thm]{Lemma}

\theoremstyle{definition}

\theoremstyle{remark} 
\newtheorem{rem}{Remark}[section]  
\def\eqref#1{(\ref{#1})} 

\newcommand {\mat}      [1] {\left[\begin{array}{#1}}
\newcommand {\rix}          {\end{array}\right]}
\newcommand {\de}      [1] {\left|\begin{array}{#1}}
\newcommand {\nt}          {\end{array}\right|}

\newcommand{\bstar}       {\begin{eqnarray*}}
\newcommand{\estar}       {\end{eqnarray*}}
\newcommand{\eqn}       {\begin{eqnarray}}
\newcommand{\enn}       {\end{eqnarray}}
\newcommand{\eq}[1]   {\begin{equation}\label{#1}}
\newcommand{\en}      {\end{equation}}

\begin{document}
\begin{titlepage}  
\title{Ion Flow Dynamics with Higher-Order Permanent Charge}  
\author[1,2]{Hamid Mofidi\thanks{\tt h.mofidi@bimsa.cn}}
\affil[1]{Beijing Institute of Mathematical Sciences and Applications (BIMSA), Beijing 101408, China}
\affil[2]{Yau Mathematical Sciences Center, Tsinghua University, Beijing 100084, China}

\date{}
\end{titlepage}

\maketitle

{\textbf{Abstract.}}
   \smallskip
In this research, we explore how permanent charges affect the movement of ionic currents through ion channels. We use a quasi-one-dimensional classical Poisson–Nernst–Planck (PNP) model to study two types of ions, one positively charged and the other negatively charged. The distribution of permanent charges is simple, with zero values at both ends and a constant permanent charge in the middle region. We treat the classical PNP model as a boundary value problem (BVP) of a singularly perturbed system.
The solution to the BVP, in the presence of a small permanent charge exhibits a regular dependence on this parameter. We conduct a systematic perturbation analysis for the singular solution, focusing on understanding higher-order effects arising from permanent charges. This analysis uncovers a subtle interplay between boundary conditions and channel geometry, determining the overall impact.
 \\


\smallskip


\vspace{5mm}

\section{Introduction.}\label{Intro}
\setcounter{equation}{0}

Ion channels are membrane proteins that allow ions to cross cell membranes, creating electric signals that control various biological processes. The structure of ion channels, which have cylindrical shapes with variable cross-sections and permanent charges from amino acid side chains, is key to their properties \cite{HH52,HK55}. The study of ion channels involves both their structure and ionic flow properties. The channel's permanent charge depends on the distribution of amino acid side chains along its cylindrical shape \cite{Eis, Eis1,MGNHEB09}.

The Poisson-Nernst-Planck (PNP) model stands out as one of the most commonly utilized mathematical frameworks for studying ion channels. This model takes into account the interplay between structural characteristics and physical parameters, and researchers have extensively examined it using a geometric singular perturbation approach. Through the application of this approach, the PNP model can be simplified into an algebraic system referred to as the governing system. Analyzing this governing system unveils crucial properties of ion channels, providing valuable insights for informed design and optimization across various applications \cite{Eis4, EL07, FLMZ22}.

The effects of permanent charge on ionic flows have been investigated by several studies using the PNP model, with both analytical and numerical methods. Liu et al. \cite{JEL19,Liu18} examined the flux ratios and ion channel structures via PNP, and analyzed how they influence the fluxes, boundary concentrations, and electric potentials of the system. Other papers explored the reversal potential and permanent charge under unequal diffusion coefficients, and derived universal properties of the system \cite{ELM19, M20, M21, M23, ML20} or numerically studied the permanent charge effects on flux ratios, revealing new phenomena and qualitative changes \cite{LTZ12,MEL20}. These studies enhance the understanding of the channel geometry and the role of permanent charge in ion channel dynamics.

This manuscript explores the influence of permanent charge on ion channel dynamics through a combination of theoretical and numerical approaches. The zeroth and first order solutions in permanent charge, outlined in \cite{JLZ15}, are revisited to facilitate higher-order analyses. Analytical expressions for the computationally intricate second order solutions are derived. Additionally, novel insights are introduced based on numerical investigations of both linear and quadratic solutions.

The paper follows this structure: Section \ref{sec-setup} introduces the classical Poisson–Nernst–Planck (PNP) model for ion channels and establishes a quasi-one-dimensional electrodiffusion model in Section \ref{sec-Quasi1}, considering two types of ions with different charges and a simple distribution of permanent charge. Section \ref{sec-dimless} transforms the model into a dimensionless form for simplified analysis. Section \ref{sec-Gov} presents the governing system for the boundary value problem (BVP). In Section \ref{sec-expansion}, the singular solutions in the presence of small permanent charge are analyzed, exploring higher-order effects. Section \ref{sec-zeroth-first} and Section \ref{sec-second} respectively delve into the zeroth, first, and second-order solutions and their implications for system behavior. Notably, Section \ref{sec-second} introduces new analytical results for second-order solutions in $Q_0$. Section \ref{sec-num} provides computational outcomes for first and second-order solutions in $Q_0$ and numerically investigates the impact of permanent charge on fluxes and I-V relations, revealing the intricate interplay between permanent charge, boundary conditions, and channel geometry. Section \ref{sec-num01} and Section \ref{sec-num2} respectively focus on first and second-order effects. Finally, Section \ref{sec-conclusion} concludes the manuscript, summarizing the main results, discussing implications, and suggesting directions for future research.

\section{PNP Systems for Ion Channels: Setup and Key Results.}\label{sec-setup}
\setcounter{equation}{0}

PNP systems, essential for studying ionic flows, originate from molecular dynamic models \cite{SNE}, Boltzmann equations \cite{Bar}, and variational principles \cite{HEL10, HLE12}. Advanced coupling with Navier–Stokes equations \cite{Bie11,CEJS95,SR81} and rigorous establishment of the Onsager reciprocal law \cite{GO68} offer sophisticated insights, striking a balance between accuracy and analytical/computational challenges, supported by reviews and model comparisons \cite{IR02,RABI04}. 

Building upon this foundation, we further streamline PNP models, especially for ion channels with narrow cross-sections relative to lengths, resulting in quasi-one-dimensional models \cite{NE98}.
This reduction yields quasi-one-dimensional models \cite{NE98}, with rigorous justification provided in \cite{LW10}. The streamlined approach addresses both accuracy and analytical/computational challenges.

This section provides a detailed exposition of our mathematical model for ionic flows, focusing on the essential setup and key results. Specifically, we explore a quasi-one-dimensional Poisson-Nernst-Planck (PNP) model that characterizes ion transport within a confined channel featuring a permanent charge. To ensure clarity in our subsequent analysis, we introduce notation and assumptions consistently used throughout the paper. Moreover, we review relevant findings from previous literature, such as \cite{EL07, Liu09}, serving as crucial foundations for our contributions outlined in the following sections.

\subsection{A Quasi-One-Dimensional PNP Model.}\label{sec-Quasi1}
Our analysis is based on a   quasi-one-dimensional PNP model first proposed in \cite{NE98} and, for a special case, rigorously justified in \cite{LW10}.
For a mixture of  $n$ ion species,
a quasi-one-dimensional    PNP  model   is
\begin{align}\begin{split}\label{ssPNP}
&\frac{1}{A(X)}\frac{d}{dX}\Big(\varepsilon_r(X)\varepsilon_0A(X)\frac{d\Phi}{dX}\Big)=-e_0\Big(\sum_{s=1}^nz_sC_s+{\mathcal{Q}}(X)\Big),\\
 & \frac{d{\mathcal{J}}_k}{dX}=0, \quad -{\mathcal{J}}_k=\frac{1}{k_BT}{\mathcal{D}}_k(X)A(X)C_k\frac{d\mu_k}{d X}, \quad
 k=1,2,\cdots, n,
\end{split}
\end{align}
where $X\in [a_0,b_0]$ is the coordinate along the axis of the channel and baths of total length $b_0-a_0$, $A(X)$ is the
 area of cross-section  of the channel over the longitudinal location $X$, $e_0$ is the elementary charge (we reserve the letter $e$ for the Euler's  number -- the base for the natural exponential function), $\varepsilon_0$ is the vacuum permittivity, $\varepsilon_r(X)$ is the relative dielectric coefficient, ${\mathcal{Q}}(X)$ is the permanent charge density, $k_B$ is the Boltzmann constant, $T$ is the absolute temperature, $\Phi$ is the electric potential,   for the $k$th ion species, $C_k$ is the concentration, $z_k$ is the valence, ${\mathcal{D}}_k(X)$ is the diffusion coefficient, $\mu_k$ is the electrochemical potential, and ${\mathcal{J}}_k$ is the flux density.

Equipped with the system (\ref{ssPNP}),  a meaningful    boundary condition  for ionic flow through ion channels (see,   \cite{EL07} for reasoning) is, for $k=1,2,\cdots, n$,
\begin{equation}
\Phi(a_0)={\mathcal{V}}, \ \ C_k(a_0)={\mathcal{L}}_k>0; \quad \Phi(b_0)=0,  \ \
C_k(b_0)={\mathcal{R}}_k>0. \label{ssBV}
\end{equation}
In relation to typical experimental designs, the positions $X=a_0$ and $X=b_0$ are located in the baths separated by the channel and are locations for two electrodes that are applied to control or drive the ionic flow through the ion channel. 
An important measurement is the  I-V (current-voltage) relation where, for fixed ${\mathcal{L}}_k$'s and ${\mathcal{R}}_k$'s, the current ${\mathcal{I}}$ depends on the transmembrane potential (voltage) ${\mathcal{V}}$ by
${\mathcal{I}}=\sum_{s=1}^nz_s{\mathcal{J}}_s({\mathcal{V}}).
$

Certainly, the relations of individual fluxes ${\mathcal{J}}_k$ with respect to ${\mathcal{V}}$ are more informative, but, measuring them experimentally  is much more difficult.
Ideally, the experimental designs should not affect the intrinsic ionic flow properties so one would like to design the boundary conditions to meet the so-called electroneutrality 
$
\sum_{s=1}^nz_s{\mathcal{L}}_s=0=\sum_{s=1}^nz_s{\mathcal{R}}_s.
$
The reason for this is that, otherwise, there will be sharp boundary layers which cause significant changes  (large gradients) of the electric potential and concentrations near the boundaries so that a measurement of these values has non-trivial uncertainties. 
One smart design to remedy this potential problem is the ``four-electrode-design": two `outer electrodes' in the baths far away from the ends of the ion channel to provide the driving force and two `inner electrodes' in the bathes near the ends of the ion channel to measure the electric potential and the concentrations as the ``real'' boundary conditions for the ionic flow. At the inner electrodes locations, the electroneutrality conditions are reasonably satisfied, and hence, the electric potential and concentrations vary slowly and a measurement of these values would be robust. 
 The cross-section area $A(X)$ typically has the property that $A(X)$ is much smaller for $X\in (a_0,b_0)$ (the neck region of the channel) than that for $X\not\in [a_0,b_0]$.

\subsection{Dimensionless Form of the Quasi-One-Dimensional PNP Model.}\label{sec-dimless}
The following rescaling or its variations have been widely used for the convenience of mathematical analysis \cite{Gil99, JL12}.   
  Let $C_0$ be a characteristic concentration of the ion solution. 
 We now make a dimensionless re-scaling of the variables in the system (\ref{ssPNP}) as follows.
\begin{align}\label{rescale}\begin{split}
&\varepsilon^2=\frac{\varepsilon_r\varepsilon_0k_BT}{e_0^2(b_0-a_0)^2C_0},\; x=\frac{X-a_0}{b_0-a_0},\;  h(x)=\frac{A(X)}{(b_0-a_0)^2},
\;  Q(x)=\frac{{\mathcal{Q}}(X)}{C_0},  \\
&D(x)={\mathcal{D}}(X),\; \phi(x)=\frac{e_0}{k_BT}\Phi(X), \; c_k(x)=\frac{C_k(X)}{C_0}, \;  
 J_k=\frac{{\mathcal{J}}_k}{(b_0-a_0)C_0   {\mathcal{D}}_k}. 
\end{split}
\end{align}
  We assume $C_0$ is fixed but large so that the parameter $\varepsilon$ is small. Note that
 $\varepsilon=\lambda_D/(b_0-a_0)$,
   where $\lambda_D$  is the Debye screening length.
   In terms of the new variables,    the BVP (\ref{ssPNP}) and (\ref{ssBV}) becomes 
\begin{align}\label{PNP2}
\begin{split}
&\frac{\varepsilon^2}{h(x)}\frac{d}{dx}\left(h(x)\frac{d\phi}{dx}\right)=-\sum_{s=1}^nz_s
c_s -Q(x),\\
&\frac{d J_k}{dx}=0, \quad -  J_k=\frac{1}{k_BT}D(x)h(x)c_k\frac{d \mu_k}{d x},
\end{split}
\end{align}
with  boundary conditions at $x=0$ and $x=1$
\begin{align}\label{BVO}
\begin{split}
\phi(0)=&V,\; c_k(0)=L_k; \;
 \phi(1)=0,\; c_k(1)=R_k,
\end{split}
\end{align}
where 
$V:=\frac{e_0}{k_BT}{\mathcal{V}},\quad L_k:=\frac{{\mathcal{L}}_k}{C_0},\quad R_k:=\frac{{\mathcal{R}}_k}{C_0}.
$
The permanent charge $Q(x)$ is
\begin{align}\label{Q}
Q(x)=\left\{\begin{array}{cc}
0, & x\in (0,a)\cup (b,1)\\
Q_0, &x\in (a,b),
\end{array}\right.
\end{align}
where 
$
0<a=\frac{A-a_0}{a_1-a_0}<b=\frac{B-a_0}{a_1-a_0}<1.
$
We will take the ideal component $\mu_k^{id}$ only for the electrochemical potential. In terms of the new variables, it becomes
\begin{equation}\label{mu-id}
\frac{1}{k_BT}\mu_k^{id}(x)=z_k\phi(x)+\ln c_k(x).
\end{equation}
The ideal component $\mu_k^{id}(x)$ contains contributions of ion particles as point charges and ignores the ion-to-ion interaction. PNP models including ideal components are referred to as classical PNP models. 
Recall that the critical assumption is that $\varepsilon$ is small. This assumption allows us to treat the BVP (\ref{PNP2}) with (\ref{BVO}) as a singularly perturbed problem. A general framework for analyzing such singularly perturbed BVPs in PNP-type systems has been developed in prior works \cite{EL07,JLZ15} for classical PNP systems and in \cite{JL12,LM22,LTZ12} for PNP systems with finite ion sizes.

 
We now recall the result in \cite{EL07} that our work will be based on.   For $n=2$ with $z_1>0>z_2$,  the authors of \cite{EL07} applied geometric singular perturbation theory to construct  the singular orbit of the BVP  (\ref{PNP2}) and  (\ref{BVO}). The BVP is then reduced to a connecting problem: finding an orbit from
$
B_0=\{(V,u,L_1,L_2,J_1,J_2, 0):\;\mbox{arbitrary } u, J_1,J_2\},
$
to
$ B_1=\{(0,u,R_1,R_2,J_1,J_2, 1):\;\mbox{arbitrary } u, J_1,J_2\}.
$

In view of the jumps of permanent charge $Q(x)$ at $x=a$ and $x=b$, the construction of singular orbits is split into three intervals $[0,a]$, $[a,b]$, $[b,1]$ as follows. To do so, one introduces (unknown) values of $(\phi,c_1,c_2)$ at $x=a$ and $x=b$:
\begin{align}\label{6unknown}
\phi(a)=\phi^a,\; c_1(a)=c_1^a,\; c_2(a)=c_2^a;\quad \phi(b)=\phi^b,\; c_1(b)=c_1^b,\; c_2(a)=c_2^b.
\end{align}

\noindent Then these values determine boundary conditions at $x=a$ and $x=b$  as 
$
B_a=\{(\phi^a, u, c_1^a, c_2^a, J_1, J_2, a):\;\mbox{arbitrary } u, J_1,J_2\},
$
and 
$ B_b=\{(\phi^b, u, c_1^b, c_2^b, J_1, J_2, b): \;\mbox{arbitrary } u, J_1,J_2\}.
$
 Then, there are six unknowns $\phi^{a}$, $\phi^b$, $c_k^a$ and $c_k^b$ for $k=1,2$ should be determined.
On each interval, a singular orbit typically consists two singular layers and one regular layer. 
\begin{description}
\item{(1)}\ On interval $[0,a]$, a singular orbit from $B_0$ to $B_a$ consists of two singular layers located at $x=0$ and $x=a$, denoted as $\Gamma_0^l$ and $\Gamma_a^l$, and one regular layer $\Lambda_l$. Furthermore, with the preassigned values $\phi^a$, $c_1^a$ and $c_2^a$, the flux $J_k^l$ and $u_l(a)$ are uniquely determined so that
$(\phi^a,u_l(a), c_1^a, c_2^a, J_1^l, J_2^l, a)\in B_a.$
\item{(2)}\ On interval $[a,b]$, a singular orbit from $B_a$ to $B_b$ consists of two singular layers located at $x=a$ and $x=b$, denoted as $\Gamma_a^r$ and $\Gamma_b^l$, and one regular layer $\Lambda_m$. Furthermore, with the preassigned values $(\phi^a, c_1^a, c_2^a)$ and $(\phi^b, c_1^b, c_2^b)$, the flux $J_k^m$, $u_m(a)$ and $u_m(b)$ are uniquely determined so that
$(\phi^a,u_m(a), c_1^a, c_2^a, J_1^m, J_2^m, a)\in B_a\;\mbox{ and }\; (\phi^b,u_m(b), c_1^b, c_2^b, J_1^m, J_2^m, b)\in B_b.
$
\item{(3)}\ On interval $[b,1]$, a singular orbit from $B_b$ to $B_1$ consists of two singular layers are located at $x=b$ and $x=1$, denoted as $\Gamma_b^r$ and $\Gamma_1^l$, and one regular layer $\Lambda_r$. Furthermore, with the preassigned values $\phi^b$, $c_1^b$ and $c_2^b$, the flux $J_k^r$ and $u_r(b)$ are uniquely determined so that
$(\phi^b,u_r(b), c_1^b, c_2^b, J_1^r, J_2^r, b)\in B_b.
$
\end{description}

\subsection{Governing System for the BVP.}\label{sec-Gov}

The matching conditions of the connecting problem in previous section are 
 \begin{align}\label{GSys}
 J_k^l=J_k^m=J_k^r\;\mbox{ for }\; k=1,2,\;   u_l(a)=u_m(a)\;\mbox{ and }\; u_m(b)=u_r(b).
 \end{align}
  There are total six conditions, which are exactly the same number of unknowns preassigned in (\ref{6unknown}). Then the singular connecting problem is reduced to {\em the governing system} (\ref{GSys}) (see \cite{EL07} for an explicit form of the governing system). More precisely, 

\begin{equation}\label{Matching}
\begin{aligned}
& z_1c_1^{a}e^{z_1(\phi^{a}-\phi^{a,m})}+ z_2c_2^{a}e^{z_2(\phi^{a}-\phi^{a,m})}+Q_0=0,\\
&z_1c_1^{b}e^{z_1(\phi^{b}-\phi^{b,m})}+ z_2c_2^{b}e^{z_2(\phi^{b}-\phi^{b,m})} +Q_0 =0,\\
&\dfrac{z_2-z_1}{z_2}c_1^{a,l}=c_1^{a}e^{z_1( \phi^{a}-\phi^{a,m})} + c_2^{a}e^{z_2( \phi^{a}-\phi^{a,m})} +Q_0(\phi^{a}-\phi^{a,m}),\\
&\dfrac{z_2-z_1}{z_2}c_1^{b,r}=c_1^{b}e^{z_1( \phi^{b}-\phi^{b,m})} + c_2^{b}e^{z_2( \phi^{b}-\phi^{b,m})} +Q_0(\phi^{b}-\phi^{b,m}),\\
&J_1 = \dfrac{c_1^L - c_1^{a,l}}{H(a)}\Big(1+\dfrac{z_1(\phi^L - \phi^{a,l})}{\ln c_1^L - \ln c_1^{a,l}} \Big)  = \dfrac{c_1^{b,r} - c_1^{R}}{H(1)- H(b)}\Big(1+\dfrac{z_1(\phi^{b,r} - \phi^{R})}{\ln c_1^{b,r} - \ln c_1^{R}} \Big),\\
&J_2 = \dfrac{c_2^L - c_2^{a,l}}{H(a)}\Big(1+\dfrac{z_2(\phi^L - \phi^{a,l})}{\ln c_2^L - \ln c_2^{a,l}} \Big)  = \dfrac{c_2^{b,r} - c_2^{R}}{H(1)- H(b)}\Big(1+\dfrac{z_2(\phi^{b,r} - \phi^{R})}{\ln c_2^{b,r} - \ln c_2^{R}} \Big),\\
&\phi^{b,m}= \phi^{a,m} - (z_1J_1 + z_2 J_2)y,\\
&c_1^{b,m} = e^{z_1z_2(J_1+J_2)y}c_1^{a,m} - \dfrac{Q_0J_1}{z_1(J_1+J_2)}\Big( 1- e^{z_1z_2(J_1+J_2)y} \Big),\\
& J_1 + J_2= - \dfrac{(z_1-z_2)(c_1^{a,m}- c_1^{b,m})+z_2Q_0(\phi^{a,m}- \phi^{b,m})}{z_2(H(b)-H(a))},
\end{aligned}
\end{equation}
where $y>0$ is also unknown, and under electroneutrality boundary conditions $z_1L_1 = -z_2L_2= L$ and $z_1R_1 = -z_2R_2= R$,
\begin{equation}\label{Matching2}
\begin{aligned}
\phi^{L} =& V, \quad \phi^{R} = 0, \quad z_1c_1^L =  -z_2c_2^L = L,  \quad z_1c_1^R =  -z_2c_2^R = R,\\
\phi^{a,l} =& \phi^a - \frac{1}{z_1-z_2}\ln \frac{-z_2c_2^a}{z_1c_1^a}, \quad \phi^{b,r} = \phi^b - \frac{1}{z_1-z_2}\ln \frac{-z_2c_2^b}{z_1c_1^b},\\
c_1^{a,l} =& \frac{1}{z_1}(z_1c_1^a)^{\frac{-z_2}{z_1-z_2}} (-z_2c_2^a)^{\frac{z_1}{z_1-z_2}}, \quad c_2^{a,l} = -\frac{1}{z_2}(z_1c_1^a)^{\frac{-z_2}{z_1-z_2}} (-z_2c_2^a)^{\frac{z_1}{z_1-z_2}},\\
c_1^{b,r} =& \frac{1}{z_1}(z_1c_1^b)^{\frac{-z_2}{z_1-z_2}} (-z_2c_2^b)^{\frac{z_1}{z_1-z_2}}, \quad c_2^{b,r} = -\frac{1}{z_2}(z_1c_1^b)^{\frac{-z_2}{z_1-z_2}} (-z_2c_2^b)^{\frac{z_1}{z_1-z_2}},\\
c_1^{a,m} =& e^{z_1( \phi^{a}-\phi^{a,m})}c_1^a, \quad c_1^{b,m} = e^{z_1( \phi^{b}-\phi^{b,m})}c_1^b,\\
H(x) =& \int_0^x \dfrac{1}{h(s)}ds.
\end{aligned}
\end{equation}

\begin{rem}\em In \eqref{Matching}, the unknowns are: $\phi^{[1]},~\phi^{[2]},~ c_1^{[1]},~c_2^{[1]},~ c_1^{[2]},~c_2^{[2]},~J_1,~\phi^{[1,+]},~\phi^{[2,-]},~y^*$ and $Q$  that is, there are eleven unknowns that matches the total number of equations on \eqref{Matching}. 
\end{rem}

\begin{rem}\em 
In the following sections we will face long terms in some formulas. For simplicity, we introduce, for $k=0,1,2$,
\begin{equation}\label{Abbrev}
\begin{aligned}
I_{k} =& z_1J_{1k}+z_2J_{2k}, \quad T_{k} = J_{1k}+ J_{2k}.
\end{aligned}
\end{equation}
\end{rem}


\section{Expanding Singular Solutions in the Presence of Small Permanent Charge.}\label{sec-expansion}
\setcounter{equation}{0}
This section, and in particular subsection \ref{sec-second}, involves numerous detailed computations, executed with a stringent and meticulous approach and verified multiple times. However, most of the computations are omitted from the text due to their extensive nature, which might hinder the readability of the paper. Interested readers are encouraged to follow each step closely and thoroughly analyze the process to reproduce the results. Furthermore, the authors are available for further clarification upon request and can provide a detailed version of the paper to the journal if needed.

Assuming that $d=0$ and $|Q_0|$ is small, we expand all unknown quantities in the governing system \eqref{Matching} and \eqref{Matching2} in $Q_0$, i.e., we write
\begin{equation}\label{k0j_expan}
\begin{aligned}
&\phi^a= \phi_{0}^a + \phi_{1}^aQ_0+ \phi_{2}^aQ_0^2 + O(Q_0^3),\hspace*{.1in} \phi^b= \phi_{0}^b + \phi_{1}^bQ_0+ \phi_{2}^bQ_0^2 + O(Q_0^3),\\
&c_k^a= c_{k0}^a + c_{k1}^aQ_0+ c_{k2}^aQ_0^2 + O(Q_0^3), \hspace*{.1in}c_k^b= c_{k0}^b + c_{k1}^bQ_0+ c_{k2}^bQ_0^2 + O(Q_0^3),\\
&y= y_0 + y_1Q_0+ y_2Q_0^2 + O(Q_0^3), \hspace*{.1in} J_k= J_{k0} + J_{k1}Q_0+ J_{k2}Q_0^2 + O(Q_0^3).
\end{aligned} 
\end{equation}

\subsection{Zeroth and first order solutions in $Q_0$  of \eqref{Matching} and \eqref{Matching2}. }\label{sec-zeroth-first}
The problem for $Q_0=0$ has been solved in \cite{Liu05} for $h(x)=1$ and, for a general $h(x)$, it can be solved as in \cite{EL07} over the interval $[0,a]$. One can also obtain the zeroth order solution directly by substituting (3.1) into \eqref{Matching}, expanding the identities in $Q_0$, and comparing the terms of like-powers in $Q_0$. We summarize the results for the zeroth order terms below. The detailed proofs for the zeroth and first order solutions are omitted and can be found in \cite{JLZ15}. Those are needed for the computational calculations in Section \ref{sec-num01} as well as the computations for second order solutions in Section \ref{sec-num2}.
 Denote,
\begin{equation}
\alpha = \dfrac{H(a)}{H(1)} \quad \text{and} \quad \beta = \dfrac{H(b)}{H(1)}.
\end{equation}

\begin{prop}\label{prop-zeroth}
The zeroth order solution in $Q_0$ of \eqref{Matching} and \eqref{Matching2}, under electroneutrality boundary conditions $z_1L_1 = -z_2L_2= L$ and $z_1R_1 = -z_2R_2= R$ where one obtains $c_j^L=L_j, c_j^R=R_j, \phi^L=V, \phi^R=0$, is given by
$$
\begin{aligned}
&z_1c_{10}^{a,l}= z_1c_{10}^{a,m}= z_1c_{10}^a= (1-\alpha) L + \alpha R, \quad z_1c_{10}^a = -z_2c_{20}^a,\\
&z_1c_{10}^{b,m}= z_1c_{10}^{b,r}= z_1 c_{10}^b= (1-\beta) L + \beta R, \quad z_1c_{10}^b = -z_2c_{20}^b,\\
&\phi_0^{a,l} = \phi_0^{a,m} = \phi_0^{a} = \dfrac{\ln \big((1-\alpha) L + \alpha R \big) - \ln R}{\ln L - \ln R }V, \\
&\phi_0^{b,m} = \phi_0^{b,r} = \phi_0^{b} =  \dfrac{\ln \big((1-\beta) L + \beta R \big) - \ln R}{\ln L - \ln R }V, \\
& y_0 = \dfrac{H(1)}{(z_1-z_2)(L-R)}\ln \dfrac{(1-\alpha) L + \alpha R }{(1-\beta) L + \beta R },\\
&J_{10} = \dfrac{L - R}{z_1H(1) (\ln L - \ln R)}(z_1V+\ln L - \ln R),\\
& J_{20} = - \dfrac{L - R}{z_2H(1) (\ln L - \ln R)}(z_2V+\ln L - \ln R).
\end{aligned}
$$
\end{prop}
To compute the first-order terms in $Q_0$, we adopt the method introduced in \cite{JLZ15}, where we represent the intermediate variables in relation to the zeroth-order terms. The proof process is straightforward: by expanding the relevant identities in \eqref{Matching2} with respect to $Q_0$, comparing the first-order terms in $Q_0$, and utilizing the results derived from Proposition \eqref{prop-zeroth}, we can establish the desired relations.
\begin{lem}\label{lem-1st}
One has
$$
\begin{aligned}
&z_1 c_{11}^a + z_2 c_{21}^a = -\frac{1}{2}, \qquad  \phi_1^{a,m} = \phi_1^a + \frac{1}{2z_1(z_1-z_2)c_{10}^a},\\
&z_1 c_{11}^b + z_2 c_{21}^b = -\frac{1}{2}, \qquad  \phi_1^{b,m} = \phi_1^b + \frac{1}{2z_1(z_1-z_2)c_{10}^b},\\
& \phi_1^{a,l} = \phi_1^a - \frac{c_{10}^a c_{21}^a - c_{20}^a c_{11}^a}{(z_1-z_2)c_{10}^a c_{20}^a}, \quad c_{11}^{a,l} = \frac{z_2(c_{11}^a + c_{21}^a)}{z_2-z_1}, \quad  c_{21}^{a,l} = \frac{z_1(c_{11}^a + c_{21}^a)}{z_1-z_2},\\
& c_{11}^{a,m} = c_{11}^a - \frac{1}{2(z_1-z_2)}, \quad  c_{11}^{b,m} = c_{11}^b - \frac{1}{2(z_1-z_2)},\\
& \phi_1^{b,r} = \phi_1^b - \frac{c_{10}^b c_{21}^b - c_{20}^b c_{11}^b}{(z_1-z_2)c_{10}^b c_{20}^b}, \quad c_{11}^{b,r} = \frac{z_2(c_{11}^b + c_{21}^b)}{z_2-z_1}, \quad  c_{21}^{b,r} = \frac{z_1(c_{11}^b + c_{21}^b)}{z_1-z_2}.
\end{aligned}
$$
\end{lem}
By applying the same procedure as above to the remaining four identities in \eqref{Matching}, and utilizing the results from Proposition \eqref{prop-zeroth} and Lemma \eqref{lem-1st}, one can directly derive the first-order terms as follows.
\begin{prop}\label{prop-first}
The first-order terms of the solution in $Q_0$ for the system \eqref{Matching} are as follows:
$$
\begin{aligned}
& c_{11}^a =\frac{z_2\alpha(\phi_0^b- \phi_0^a)}{z_1-z_2} - \frac{1}{2(z_1-z_2)}, \qquad c_{21}^a =\frac{z_1\alpha(\phi_0^b- \phi_0^a)}{z_2-z_1} - \frac{1}{2(z_2-z_1)},\\
& c_{11}^b =\frac{z_2(1-\beta)(\phi_0^a- \phi_0^b)}{z_1-z_2} - \frac{1}{2(z_1-z_2)}, \quad c_{21}^b =\frac{z_1(1-\beta)(\phi_0^a- \phi_0^b)}{z_2-z_1} - \frac{1}{2(z_2-z_1)},\\
&\phi_1^a = \frac{(1+z_1\lambda)(1+z_2\lambda)(c_{10}^b - c_{10}^a)(\ln c_1^L - \ln c_{10}^a)}{z_1(z_1-z_2)c_{10}^a c_{10}^b (\ln c_1^R - \ln c_1^L)} + \frac{1}{2z_1(z_1-z_2)c_{10}^a} + \frac{z_2\alpha (\phi_0^b - \phi_0^a)}{(z_1-z_2)c_{10}^a}\lambda,
\end{aligned}
$$
$$
\begin{aligned}
&\phi_1^b = \frac{(1+z_1\lambda)(1+z_2\lambda)(c_{10}^b - c_{10}^a)(\ln c_1^R - \ln c_{10}^b)}{z_1(z_1-z_2)c_{10}^a c_{10}^b (\ln c_1^R - \ln c_1^L)} + \frac{1}{2z_1(z_1-z_2)c_{10}^b} + \frac{z_2(1-\beta) (\phi_0^a - \phi_0^b)}{(z_1-z_2)c_{10}^b}\lambda,\\
& y_1 = \frac{\Big((1-\beta)c_1^L + \alpha c_1^R \Big)(\phi_0^a -\phi_0^b)}{z_1(z_1-z_2)T_0 c_{10}^a c_{10}^b} +  \frac{(\ln c_{10}^a - \ln c_{10}^b)(\phi_0^a - \phi_0^b)}{z_1(z_1-z_2)T_0 (c_1^L- c_1^R)} - \frac{(z_2J_{10}+z_1J_{20})(c_{10}^a - c_{10}^b)}{z_1^2z_2(z_1-z_2)T_0^2 c_{10}^a c_{10}^b},
\end{aligned}
$$
$$
\begin{aligned}
&\hspace*{-2.45in}J_{11}= \dfrac{A\Big(z_2(1-B)V + \ln L - \ln R \Big)}{(z_1-z_2) H(1) (\ln L - \ln R)^2}(z_1V+\ln L - \ln R), \\
&\hspace*{-2.45in}J_{21}= \dfrac{A\Big(z_1(1-B)V + \ln L - \ln R \Big)}{(z_2-z_1) H(1) (\ln L - \ln R)^2}(z_2V+\ln L - \ln R),
\end{aligned}
$$
where, under electroneutrality boundary conditions $z_1L_1 = -z_2L_2= L$ and $z_1R_1 = -z_2R_2= R$, and in terms of $\alpha =\frac{H(a)}{H(1)}$ and $\beta =\frac{H(b)}{H(1)}$, the expressions for $A$, $B$ and $\lambda$ are 
\begin{equation}\label{AB-LR}
\begin{aligned}
&A=A(L,R) = - \dfrac{(\beta - \alpha)(L-R)^2}{\Big( (1-\alpha)L + \alpha R \Big)\Big( (1-\beta)L + \beta R \Big)(\ln L - \ln R)},\\
&B=B(L,R) =  \frac{1}{A}\ln \dfrac{ (1-\beta)L + \beta R }{(1-\alpha)L + \alpha R},\qquad
 \lambda = \lambda(L,R) = \frac{V}{\ln L - \ln R}.
\end{aligned}
\end{equation}
\end{prop}

\subsection{Second order solutions in $Q_0$ of \eqref{Matching} and \eqref{Matching2}.}\label{sec-second}
The results presented in this section extend the findings of the previous section, employing a consistent approach and methodologies with solutions exhibiting regularity concerning the permanent charge. Nevertheless, it is crucial to note that certain intricate calculations, owing to their extensive nature, have been omitted for the sake of clarity in the presentation. 
For the second order terms in $Q_0$, we will first express the intermediate variables such as $\phi_{2}^{a,l},~c_{k2}^{a,l}$, etc. in terms of zeroth and first order terms and $\phi_{2}^a,~c_{k2}^a$, etc.
\begin{lem}\label{LemQuad1}
One has
$$
\begin{aligned}
 &z_1c_{12}^a + z_2c_{22}^a =  -\dfrac{z_1+z_2}{{24z_1(z_1-z_2)c_{10}^a}},\hspace*{.15in} \phi_{2}^a - \phi_{2}^{a,m} = \dfrac{z_1^2c_{11}^a + z_2^2c_{21}^a}{2\big(z_1(z_1-z_2)c_{10}^a\big)^2}  - \dfrac{ z_1+z_2}{12\big(z_1(z_1-z_2)c_{10}^a\big)^2},\\
 &z_1c_{12}^b + z_2c_{22}^b =  -\dfrac{z_1+z_2}{{24z_1(z_1-z_2)c_{10}^b}},\hspace*{.15in} \phi_{2}^b - \phi_{2}^{b,m} = \dfrac{z_1^2c_{11}^b + z_2^2c_{21}^b}{2\big(z_1(z_1-z_2)c_{10}^b\big)^2}  - \dfrac{ z_1+z_2}{12\big(z_1(z_1-z_2)c_{10}^b\big)^2}.
\end{aligned}
$$
\end{lem}
\begin{proof}
We present the derivations of the first two equations without showing the tedious computations, which mainly involve manipulating lengthy terms. The first step is to substitute  \eqref{k0j_expan} into the first equation in \eqref{Matching} and expand with respect to the parameter $Q$.
Then, by applying a Taylor expansion for the function $e^{z_k(\phi^a - \phi^{a,m})}$ with respect to $Q$,  we obtain the following expression for the second-order terms:
\begin{equation}\label{phi2^a}
\begin{aligned}
\phi_{2}^a- \phi_{2}^{a,m} 
& = - \dfrac{z_1c_{12}^a + z_2c_{22}^a}{z_1(z_1-z_2)c_{10}^a} + \dfrac{z_1^2c_{11}^a + z_2^2c_{21}^a}{2\big(z_1(z_1-z_2)c_{10}^a\big)^2} - \dfrac{z_1^3c_{10}^a + z_2^3c_{20}^a}{8\big(z_1(z_1-z_2)c_{10}^a\big)^3}.
\end{aligned}
\end{equation}
Next, we substitute the expression for $c_1^{a,l}$ from $c_1^{a,l}$ from \eqref{Matching2} into the third equation of \eqref{Matching} and expand the resulting equation up to third-order terms in $Q$, which gives us:
\begin{equation}\label{c1a-rep}
\dfrac{z_2 -z_1}{z_2}\Big(\frac{1}{z_1}(z_1c_1^a)^{\frac{-z_2}{z_1-z_2}}(-z_2c_2^a)^{\frac{z_1}{z_1-z_2}} \Big) = c_1^a e^{z_1(\phi^a - \phi^{a,m})} + c_2^a e^{z_2(\phi^a - \phi^{a,m})} + Q (\phi^a - \phi^{a,m}).
\end{equation}
To obtain the desired result, we must carefully compute the expansions on both sides of  \eqref{c1a-rep} up to the third order and simplify the terms accordingly.
to obtain the desired result.
\end{proof}

Now, we  expand the relevant identities in \eqref{Matching2} in $Q$ of the paper, compare the first and second order terms in $Q$ and use the results for the zeroth and first order terms in Prop 3.1 and Lemma 3.3.

\begin{lem}\label{LemQuad2} One has
$$
\begin{aligned}
&\phi_{2}^{a,l} = \phi_{2}^a + \dfrac{z_1z_2\alpha(\phi_{0}^b - \phi_{0}^a)}{2\big(z_1(z_1-z_2)c_{10}^a\big)^2}-\dfrac{z_1+z_2}{6\big(z_1(z_1-z_2)c_{10}^a\big)^2} , \quad c_{12}^{a,l}= \dfrac{z_2(c_{12}^a + c_{22}^a)}{z_2-z_1} + \dfrac{z_2}{8z_1c_{10}^a(z_1-z_2)^2} , \\
&\phi_{2}^{b,r} =  \phi_{2}^b + \dfrac{z_1z_2(1-\beta)(\phi_{0}^a - \phi_{0}^b)}{2\big(z_1(z_1-z_2)c_{10}^b\big)^2}-\dfrac{z_1+z_2}{6\big(z_1(z_1-z_2)c_{10}^b\big)^2} , \quad c_{12}^{b,r}= \dfrac{z_2(c_{12}^b + c_{22}^b)}{z_2-z_1} + \dfrac{z_2}{8z_1c_{10}^b(z_1-z_2)^2} ,\\
&c_{22}^{a,l}= \dfrac{z_1(c_{12}^a + c_{22}^a)}{z_1-z_2} - \dfrac{z_1}{8z_1c_{10}^a(z_1-z_2)^2}  , \qquad c_{22}^{b,r}= \dfrac{z_1(c_{12}^b + c_{22}^b)}{z_1-z_2} - \dfrac{z_1}{8z_1c_{10}^b(z_1-z_2)^2}  \\
& c_{12}^{a,m} = c_{12}^a+\dfrac{z_1- 8z_2}{24z_1(z_1-z_2)^2c_{10}^a}, \qquad \qquad c_{12}^{b,m} =  c_{12}^b+\dfrac{z_1- 8z_2}{24z_1(z_1-z_2)^2c_{10}^b}.
\end{aligned}
$$
\end{lem}
\begin{proof}
Starting from the second line of \eqref{Matching2}, we can derive the second-order terms as follows:
$$
\begin{aligned}
\phi_{2}^{a,l} = & \phi_{2}^a + \frac{12z_1(z_1-z_2)c_{11}^a + 2(z_1-2z_2)}{24\left(z_1(z_1-z_2) c_{10}^a\right)^2}.
\end{aligned}
$$
By substituting $c_{11}^a$ from Proposition \eqref{prop-first}, we obtain the formula for $\phi_{2}^{a,l}$.

Moving on to the fourth line of \eqref{Matching2}, the second-order terms can be expressed as:
$$
\begin{aligned}
c_{12}^{a,l}  = & \frac{z_2(c_{12}^a + c_{22}^a)}{z_2-z_1} + \frac{z_2}{8z_1c_{10}^a(z_1-z_2)^2}.
\end{aligned}
$$
Finally, from the sixth line of \eqref{Matching2}, we can determine $c_{12}^{a,m}$. Similar relations can be found for the other terms.
\end{proof}

\begin{rem}\em

It is important to recognize that for small values of $Q_0$, we can make an approximation: $z_1c_1^a + z_2c_2^a \approx z_1c_{10}^a + z_2c_{20}^a = 0$, which implies that $\dfrac{-z_2c_2^a}{z_1c_1^a} \approx 1$. In the proof provided earlier, we applied the Maclaurin expansion of the natural logarithm, given by $\ln (x) = \ln (1 + (x-1) ) = (x-1) - \frac{1}{2}(x-1)^2 +\cdots$. This expansion converges when $|x-1| < 1.$ 

\end{rem}


By following the previously outlined procedure for the last four identities in \eqref{Matching} and leveraging the results from Proposition \eqref{prop-first}, along with Lemmas \eqref{LemQuad1} and \eqref{LemQuad2}, one can straightforwardly derive the following Lemma.

\begin{lem}\label{LemJ2} Second order fluxes of the solution in $Q_0$ to the system \ref{Matching} are given by
\begin{equation*}
\begin{aligned}
J_{12}= &  \dfrac{z_2(c_{12}^a + c_{22}^a)}{(z_1-z_2)\alpha H(1)}\Big(1+ \dfrac{z_1(\phi^L-\phi_{0}^a)}{\ln c_1^L - \ln c_{10}^a} - \dfrac{z_1(\phi^L-\phi_{0}^a)(c_1^L-c_{10}^a)}{(\ln c_1^L - \ln c_{10}^a)^2c_{10}^a}  \Big)\\
& - \dfrac{z_1(c_1^L-c_{10}^a) }{\alpha H(1)(\ln c_1^L - \ln c_{10}^a)}\Bigg( \phi_{2}^a + \dfrac{z_1z_2\alpha(\phi_{0}^b - \phi_{0}^a)}{2\big(z_1(z_1-z_2)c_{10}^a\big)^2}-\dfrac{z_1+z_2}{6\big(z_1(z_1-z_2)c_{10}^a\big)^2}\\
&  - \dfrac{z_2(\phi^L - \phi_{0}^a)}{8z_1(z_1-z_2)^2(\ln c_1^L - \ln c_{10}^a)(c_{10}^a)^2} \Bigg)\\
 & - \dfrac{  z_1 z_2(c_{11}^a+c_{21}^a)}{\alpha H(1)(\ln c_1^L - \ln c_{10}^a)(z_1-z_2)}\Bigg(\phi_{1}^a - \dfrac{(c_1^L -c_{10}^a)\phi_{1}^a}{(\ln c_1^L - \ln c_{10}^a)c_{10}^a} -\dfrac{1}{2z_1(z_1-z_2)c_{10}^a}\\
&  + \dfrac{c_1^L-c_{10}^a}{2z_1(z_1-z_2)(\ln c_1^L - \ln c_{10}^a)(c_{10}^a)^2}  + \dfrac{z_2(c_{11}^a+c_{21}^a)(\phi^L-\phi_{0}^a)(c_1^L+c_{10}^a)}{2(z_1-z_2)(\ln c_1^L - \ln c_{10}^a)(c_{10}^a)^2}\Bigg)\\
& - \dfrac{z_1z_2 (\phi^L-\phi_{0}^a)}{8z_1(z_1-z_2)^2c_{10}^a\alpha H(1)(\ln c_1^L - \ln c_{10}^a)}- \dfrac{z_2}{8z_1c_{10}^a(z_1-z_2)^2\alpha H(1)},
\end{aligned}
\end{equation*}

\begin{equation*}
\begin{aligned}
J_{22} = &  -\dfrac{z_1(c_{12}^a + c_{22}^a)}{(z_1-z_2)\alpha H(1)}\Big(1+ \dfrac{z_2(\phi^L-\phi_{0}^a)}{\ln c_2^L - \ln c_{20}^a} - \dfrac{z_2(\phi^L-\phi_{0}^a)(c_2^L-c_{20}^a)}{(\ln c_2^L - \ln c_{20}^a)^2c_{20}^a}  \Big)\\
& - \dfrac{z_2(c_2^L-c_{20}^a) }{\alpha H(1)(\ln c_2^L - \ln c_{20}^a)}\Bigg( \phi_{2}^a + \dfrac{z_1z_2\alpha(\phi_{0}^b - \phi_{0}^a)}{2\big(z_1(z_1-z_2)c_{10}^a\big)^2}-\dfrac{z_1+z_2}{6\big(z_1(z_1-z_2)c_{10}^a\big)^2}\\
&  + \dfrac{z_1(\phi^L - \phi_{0}^a)}{8z_1(z_1-z_2)^2(\ln c_2^L - \ln c_{20}^a)c_{10}^ac_{20}^a} \Bigg)\\
 & + \dfrac{  z_1 z_2(c_{11}^a+c_{21}^a)}{\alpha H(1)(\ln c_2^L - \ln c_{20}^a)(z_1-z_2)}\Big(\phi_{1}^a - \dfrac{(c_2^L -c_{20}^a)\phi_{1}^a}{(\ln c_2^L - \ln c_{20}^a)c_{20}^a} -\dfrac{1}{2z_1(z_1-z_2)c_{10}^a}\\
& + \dfrac{c_2^L-c_{20}^a}{2z_1(z_1-z_2)(\ln c_2^L - \ln c_{20}^a)(c_{10}^a)(c_{20}^a)}  - \dfrac{z_1(c_{11}^a+c_{21}^a)(\phi^L-\phi_{0}^a)(c_2^L+ c_{20}^a)}{2(z_1-z_2)(\ln c_2^L - \ln c_{20}^a)(c_{20}^a)^2}\Big)\\
&  + \dfrac{z_1z_2 (\phi^L-\phi_{0}^a)}{8z_1(z_1-z_2)^2c_{10}^a\alpha H(1)(\ln c_2^L - \ln c_{20}^a)}+ \dfrac{z_1}{8z_1c_{10}^a(z_1-z_2)^2\alpha H(1)},
\end{aligned}
\end{equation*}
where,
$$
K_1 =  T_0y_1 + T_1y_0, \ \ \
K_2 = T_2y_0+ T_1y_1 + T_0y_2.
$$
\end{lem}
\begin{proof} 
From \eqref{Matching}, we get
$
J_{1} = \frac{c_1^L - c_1^{a,l}}{H(a)} \big( 1 + \frac{z_1(\phi^L - \phi^{a,l})}{\ln c_1^L - \ln c_1^{a,l}} \big).
$
Then, expanding $\frac{c_1^L - c_1^{a,l}}{H(a)}$ with respect to the parameter $Q_0$ and noting that
$
\ln c_{1}^{a,l} = \ln c_{10}^a + \frac{c_{11}^{a,l}}{c_{10}^a}Q + \frac{2c_{12}^{a,l}c_{10}^a - (c_{11}^{a,l})^2}{2(c_{10}^a)^2}Q^2 + O(Q^3),
$
we can express $J_1$ in terms of the parameter $Q_0$. Thus, we can find the second-order term in $Q_0$ by substituting $c_{11}^{a,l}$, $c_{12}^{a,l}$, and $\phi_{1}^{a,l}$, which will lead to the first formula for $J_{12}$.
Similarly, from the second formula for $J_{1}$, we obtain,
$
J_1 = \frac{c_1^{b,r} - c_1^R}{H(1) - H(b)} \big( 1 + \frac{z_1(\phi^{b,r} - \phi^{R})}{\ln c_1^{b,r} - \ln c_1^{R}} \big).
$
We can then determine the second-order terms by substituting $c_{11}^{b,r}$, $c_{12}^{b,r}$, and $\phi_{1}^{b,r}$. This allows us to directly derive the formula for $J_{12}$.
To find the second-order terms of $c_1^b$, we first need to expand $\exp\big(z_1z_2(J_1+J_2)y\big)
$ as per \eqref{LemQuad2} and \eqref{Matching2}. Then, from the zeroth-order terms in the ninth line of \eqref{Matching}, we can obtain the expansion of $\exp\big({z_1z_2(J_1+J_2)y}\big)$ in terms of the parameter $Q_0$.
Now, by expanding $c_1^{b,m} = \exp\big({z_1z_2(J_1+J_2)y} c_1^{a,m}\big)$ from equation \eqref{Matching} and using Lemma \eqref{LemQuad2}, we can derive the expression for $c_{12}^b$. Finally, from the last equation of \eqref{Matching}, and using Lemmas \eqref{lem-1st} and \eqref{LemQuad2}, we can determine the expression for $T_2$.

\end{proof}

\begin{prop}\label{prop-cterms}
Second order intermediate concentration terms of the solution in $Q_0$ to the system \ref{Matching} are given by
\begin{equation*}
\begin{aligned}
  c_{12}^a = & - \dfrac{z_1+4z_2}{24z_1(z_1-z_2)^2c_{10}^a } - \dfrac{(\phi_{1}^{a} - \phi_{1}^{b})\alpha z_2}{(z_1-z_2)},\\
 c_{22}^a = &  \dfrac{4z_1+z_2}{24z_1(z_1-z_2)^2c_{10}^a } + \dfrac{(\phi_{1}^{a} - \phi_{1}^{b})\alpha z_1}{(z_1-z_2)},\\
  c_{12}^b = & - \dfrac{  z_1+4z_2}{24z_1(z_1-z_2)^2c_{10}^b } + \dfrac{(\phi_{1}^{a} - \phi_{1}^{b})(1-\beta) z_2}{(z_1-z_2)},\\
  c_{22}^b = &  \dfrac{4z_1+z_2}{24z_1(z_1-z_2)^2c_{10}^b } - \dfrac{(\phi_{1}^{a} - \phi_{1}^{b})(1-\beta) z_1}{(z_1-z_2)}, \\
 y_2 = &  \dfrac{(\phi_{1}^{a} - \phi_{1}^{b})y_0}{ H(1)T_0} - \frac{y_1}{c_{10}^a} \Big(\dfrac{z_2 \alpha (\phi_{0}^b-\phi_{0}^a)}{z_1-z_2} -  \dfrac{c_{10}^a(\phi_{0}^a - \phi_{0}^b)}{ H(1)T_0}- \dfrac{1}{z_1-z_2}\Big)\\
& + \dfrac{1}{2z_1^2(z_1-z_2)^2T_0}\Big(\dfrac{1}{\big(c_{10}^a\big)^2} - \dfrac{1}{\big(c_{10}^b\big)^2} \Big)   + \dfrac{ (\phi_{1}^{a} - \phi_{1}^{b})}{z_1(z_1-z_2)T_0} \Big( \dfrac{\alpha}{c_{10}^a} +  \dfrac{1-\beta}{c_{10}^b} \Big) \\
 &  - \frac{z_1z_2}{2T_0}\Big( T_0y_1 + T_1y_0 \Big)^2 + \frac{ (\phi_{0}^a - \phi_{0}^b)y_0}{ H(1)T_0c_{10}^a} \Big(\dfrac{z_2 \alpha (\phi_{0}^b-\phi_{0}^a)}{z_1-z_2} - \dfrac{1}{z_1-z_2}\Big)\\
 &  +\dfrac{J_{11}}{z_1^2z_2T_0^2} \Big(\dfrac{1}{c_{10}^b} - \dfrac{1 }{c_{10}^a} \Big) + \dfrac{  J_{10}(\phi_{0}^a - \phi_{0}^b) }{z_1^2z_2T_0^3H(1)} \Big(\dfrac{1}{c_{10}^b} - \dfrac{1 }{c_{10}^a} \Big).
\end{aligned}
\end{equation*}
\end{prop}

\begin{proof}
Initially, we start by adding up the expressions for $J_{12}$ and $J_{22}$ as outlined in the equations for $J_{12}$ and $J_{22}$ in Lemma \ref{LemJ2}, using careful simplification procedures. Afterward, we include $c_{22}^a$ and $c_{22}^b$ into the derived expression using the relevant expressions from Lemma \eqref{LemQuad1}. Through a thorough computational analysis, we determine the expressions for $c_{12}^a$ and $c_{22}^a$.

In the process of determining the variable $y_2$, our initial step involves solving the equation for $c_{12}^{b}$ as presented in Lemma \eqref{LemJ2}, specifically for $K_2$. Following this, we proceed to substitute the expressions for $K_1$ and $K_2$ and subsequently solve the equation for $y_2$, resulting in a simplified expression that provides the formula for $y_2$.
\end{proof}

Utilizing the procedure outlined earlier, we shall extend our analysis to encompass the remaining four identities specified in Equation \eqref{Matching}. With the foundational insights obtained from Proposition \eqref{prop-zeroth} and Lemma \eqref{lem-1st}, we can proceed to systematically deduce the second-order terms as delineated below.

\begin{prop}\label{prop-lastJ}
Under the electroneutrality boundary conditions, where $\phi^L=V, \phi^R=0$, $z_lL_1=-z_2L_2=L$ and $z_lR_1=-z_2R_2=R$, the following results hold,
$$
\begin{aligned}
&J_{12} = \dfrac{z_1z_2(\phi_1^a-\phi_1^b)}{ H(1)(z_1-z_2)}\Big(\dfrac{1}{z_1}+ \dfrac{(V-\phi_0^a)}{\ln c_1^L - \ln c_{10}^a} - \dfrac{(V-\phi_0^a)(c_1^L-c_{10}^a)}{(\ln c_1^L - \ln c_{10}^a)^2c_{10}^a}  \Big)\\
& \hspace*{.4in}- \dfrac{z_1(c_1^L-c_{10}^a) }{\alpha H(1)(\ln c_1^L - \ln c_{10}^a)}\Big( \phi_2^a + \dfrac{z_1z_2\alpha(\phi_0^b - \phi_0^a)}{2\big(z_1(z_1-z_2)c_{10}^a\big)^2}-\dfrac{(z_1+z_2)}{6\big(z_1(z_1-z_2)c_{10}^a\big)^2}\Big)\\
 &\hspace*{.4in} - \dfrac{  z_1 z_2(\phi_0^a-\phi_0^b)}{ H(1)(\ln c_1^L - \ln c_{10}^a)(z_1-z_2)^2}\Bigg((z_1-z_2)\phi_1^a - \dfrac{(z_1-z_2)(c_1^L -c_{10}^a)\phi_1^a}{(\ln c_1^L - \ln c_{10}^a)c_{10}^a} -\dfrac{1}{2z_1c_{10}^a}\\
& \hspace*{1.5in} + \dfrac{(c_1^L-c_{10}^a)}{2z_1(\ln c_1^L - \ln c_{10}^a)(c_{10}^a)^2} + \dfrac{z_2(c_{11}^a+c_{21}^a)(V-\phi_0^a)(c_1^L+c_{10}^a)}{2(\ln c_1^L - \ln c_{10}^a)(c_{10}^a)^2}\Bigg),
\end{aligned}
$$
\begin{equation*}
\begin{aligned}
J_{22}&= \dfrac{z_1z_2(\phi_1^a-\phi_1^b)}{ H(1)(z_2-z_1)}\Big( \dfrac{1}{z_2} + \dfrac{(V-\phi_0^a)}{\ln c_1^L - \ln c_{10}^a} - \dfrac{(V-\phi_0^a)(c_1^L-c_{10}^a)}{(\ln c_1^L - \ln c_{10}^a)^2c_{10}^a}  \Big)\\
&\hspace*{.4in}+ \dfrac{z_1(c_1^L-c_{10}^a) }{\alpha H(1)(\ln c_1^L - \ln c_{10}^a)}\Big( \phi_2^a + \dfrac{z_1z_2\alpha(\phi_0^b - \phi_0^a)}{2\big(z_1(z_1-z_2)c_{10}^a\big)^2}-\dfrac{(z_1+z_2)}{6\big(z_1(z_1-z_2)c_{10}^a\big)^2}\Big)\\
 &\hspace*{.4in} + \dfrac{  z_1 z_2(\phi_0^a-\phi_0^b)}{ H(1)(\ln c_1^L - \ln c_{10}^a)(z_1-z_2)^2}\Bigg((z_1-z_2)\phi_1^a - \dfrac{(z_1-z_2)(c_1^L -c_{10}^a)\phi_1^a}{(\ln c_1^L - \ln c_{10}^a)c_{10}^a} -\dfrac{1}{2z_1c_{10}^a}\\
& \hspace*{1.5in} + \dfrac{(c_1^L-c_{10}^a)}{2z_1(\ln c_1^L - \ln c_{10}^a)(c_{10}^a)^2} + \dfrac{z_2(c_{11}^a+c_{21}^a)(V-\phi_0^a)(c_1^L+c_{10}^a)}{2(\ln c_1^L - \ln c_{10}^a)(c_{10}^a)^2}\Bigg),
\end{aligned}
\end{equation*}

\begin{equation*}
\begin{aligned}
&\phi_{2}^a = \big(\mathcal{B}_1 \mathcal{C} - (z_1- z_2)y_0 \mathcal{B}_1 \mathcal{A}_2 - z_2 y_0 \mathcal{B}_1 \dfrac{(\phi_{1}^{b} - \phi_{1}^{a})}{ H(1)}  + \mathcal{B}_2 - \mathcal{A}_2\big)/ \big(\mathcal{A}_1 - \mathcal{B}_1 +(z_1-z_2)y_0 \mathcal{A}_1 B_1\big), \\
 &\phi_{2}^b = \Big(1 -(z_1-z_2)y_0 \mathcal{A}_1\Big) \phi_{2}^a + \mathcal{C} - (z_1- z_2)y_0 \mathcal{A}_2 - z_2 y_0 \dfrac{(\phi_{1}^{b} - \phi_{1}^{a})}{ H(1)}, 
\end{aligned}
\end{equation*}
where,
\begin{equation*}
\begin{aligned}
\mathcal{A}_1 = & - \dfrac{z_1(c_1^L-c_{10}^a) }{\alpha H(1)(\ln c_1^L - \ln c_{10}^a)}, \quad \quad   \mathcal{B}_1 =\dfrac{z_1(c_{10}^b - c_1^R) }{(1-\beta) H(1)(\ln c_{10}^{b} - \ln c_{1}^R)},\\
\mathcal{A}_2 =& \dfrac{z_1z_2(\phi_{1}^a-\phi_{1}^b)}{(z_1-z_2) H(1)}\Big(\dfrac{1}{z_1}+ \dfrac{(V-\phi_{0}^a)}{\ln c_1^L - \ln c_{10}^a} - \dfrac{(V-\phi_{0}^a)(c_1^L-c_{10}^a)}{(\ln c_1^L - \ln c_{10}^a)^2c_{10}^a}  \Big)\\
& - \dfrac{z_1(c_1^L-c_{10}^a) }{\alpha H(1)(\ln c_1^L - \ln c_{10}^a)}\Big( \dfrac{z_1z_2\alpha(\phi_{0}^b - \phi_{0}^a)}{2\big(z_1(z_1-z_2)c_{10}^a\big)^2}-\dfrac{(z_1+z_2)}{6\big(z_1(z_1-z_2)c_{10}^a\big)^2}\Big)\\
 & - \dfrac{  z_1 z_2(\phi_{0}^a-\phi_{0}^b)}{ H(1)(\ln c_1^L - \ln c_{10}^a)(z_1-z_2)^2}\Bigg((z_1-z_2)\phi_{1}^a - \dfrac{(z_1-z_2)(c_1^L -c_{10}^a)\phi_{1}^a}{(\ln c_1^L - \ln c_{10}^a)c_{10}^a} -\dfrac{1}{2z_1c_{10}^a}\\
& \hspace*{1.2in} + \dfrac{(c_1^L-c_{10}^a)}{2z_1(\ln c_1^L - \ln c_{10}^a)(c_{10}^a)^2} + \dfrac{z_2(c_{11}^a+c_{21}^a)(V-\phi_{0}^a)(c_1^L+c_{10}^a)}{2(\ln c_1^L - \ln c_{10}^a)(c_{10}^a)^2}\Bigg),
\end{aligned}
\end{equation*}

\begin{equation*}
\begin{aligned}
\mathcal{B}_2 =& \dfrac{z_1z_2 (\phi_{1}^a-\phi_{1}^b)}{(z_1-z_2)H(1)}\Big(\frac{1}{z_1}+ \dfrac{\phi_{0}^b}{\ln c_{10}^{b} - \ln c_{1}^R} - \dfrac{\phi_{0}^b(c_{10}^b-c_1^R)}{(\ln c_{10}^{b} - \ln c_{1}^R)^2c_{10}^b} \Big) \\
& +\dfrac{z_1(c_{10}^b - c_1^R) }{(1-\beta) H(1)(\ln c_{10}^{b} - \ln c_{1}^R)}\Big(  \dfrac{z_1z_2(1-\beta) (\phi_{0}^a- \phi_{0}^b) }{2\big(z_1(z_1-z_2)c_{10}^b\big)^2}- \dfrac{(z_1+ z_2)}{6\big(z_1(z_1-z_2)c_{10}^b\big)^2}\Big)\\
& - \dfrac{  z_1 z_2(\phi_{0}^b-\phi_{0}^a)}{ H(1)(\ln c_{10}^{b} - \ln c_{1}^R)(z_1-z_2)^2}\Bigg((z_1-z_2)\phi_{1}^b - \dfrac{(z_1-z_2)(c_{10}^b - c_1^R)\phi_{1}^b}{(\ln c_{10}^{b} - \ln c_{1}^R)c_{10}^b} -\dfrac{1}{2z_1c_{10}^b}\\
& \hspace*{1.1in} + \dfrac{(c_{10}^b - c_1^R)}{2z_1(\ln c_{10}^{b} - \ln c_{1}^R)(c_{10}^b)^2} + \dfrac{z_2(c_{11}^b+c_{21}^b)\phi_{0}^b(c_1^R+c_{10}^b)}{2(\ln c_{10}^{b} - \ln c_{1}^R)(c_{10}^b)^2}\Bigg),
\end{aligned}
\end{equation*}
and,
\begin{equation*}
\begin{aligned}
\mathcal{C} =& - \dfrac{z_1^2c_{11}^a+z_2^2c_{21}^a}{2\big(z_1(z_1-z_2)c_{10}^a \big)^2} + \dfrac{z_1^2c_{11}^b+z_2^2c_{21}^b}{2\big(z_1(z_1-z_2)c_{10}^b \big)^2} + \dfrac{(z_1+z_2)\big((c_{10}^b)^2 - (c_{10}^a)^2\big)}{12\big(z_1(z_1-z_2)c_{10}^ac_{10}^b \big)^2}-   I_1y_1\\
  & + \frac{(z_1-z_2)(l-r)V y_1}{ H(1)(\ln l - \ln r)c_{10}^a} \Big(\dfrac{z_2 \alpha (\phi_{0}^b-\phi_{0}^a)}{z_1-z_2} -  \dfrac{c_{10}^a(\phi_{0}^a - \phi_{0}^b)}{ H(1)T_0}- \dfrac{1}{z_1-z_2}\Big)\\
  &+ \dfrac{z_2V}{2z_1(z_1-z_2)^2(\ln l - \ln r)}\Big(\dfrac{1}{(c_{10}^a)^2} - \dfrac{1}{(c_{10}^b)^2} \Big) \\
&+ \dfrac{z_1z_2(\phi_{1}^{a} - \phi_{1}^{b})V}{(\ln l - \ln r)}\Big( \dfrac{1}{z_1(z_1-z_2)} \big( \dfrac{\alpha}{c_{10}^a} +  \dfrac{1-\beta}{c_{10}^b} \big) +\dfrac{y_0}{ H(1)}\Big)\\
& - \frac{z_1^2z_2^2V}{2(\ln l - \ln r)}\big( T_0y_1 + T_1y_0 \big)^2 + \frac{z_1z_2V(\phi_{0}^a - \phi_{0}^b)y_0}{ H(1)c_{10}^a(\ln l - \ln r)} \Big(\dfrac{z_2 \alpha (\phi_{0}^b-\phi_{0}^a)}{z_1-z_2} - \dfrac{1}{z_1-z_2}\Big)\\
 & +\dfrac{J_{11}V}{z_1T_0(\ln l - \ln r)} \Big(\dfrac{1}{c_{10}^b} - \dfrac{1 }{c_{10}^a} \Big) + \dfrac{J_{10}(\phi_{0}^a - \phi_{0}^b)V }{z_1T_0^2H(1)(\ln l - \ln r)} \Big(\dfrac{1}{c_{10}^b} - \dfrac{1 }{c_{10}^a} \Big).
\end{aligned}
\end{equation*}
Furthermore, $z_1 c_1^L= z_1 L_1= L, z_1 c_1^R= z_1 R_1= R $ due to electroneutrality and $T_0, T_1$ were defined in \ref{Abbrev}.
\end{prop}

\begin{proof} 
Starting from the expressions for $J_{12}$ and $J_{22}$ derived in Lemma \ref{LemJ2} and employing the relationships established in Lemma \ref{LemQuad2} and Proposition \ref{prop-cterms}, and through meticulous computations, one can directly derive the second-order terms for fluxes and electric potentials.
\end{proof}

\begin{rem}\em 
In Proposition \eqref{prop-lastJ},  it is noteworthy that the following relationships hold:
$$
J_{12} = \mathcal{A}_1 \phi_{2}^a + \mathcal{A}_2 = \mathcal{B}_1 \phi_{2}^b + \mathcal{B}_2, \qquad   J_{22} = -\mathcal{A}_1 \phi_{2}^a + \mathcal{A}_3 = -\mathcal{B}_1 \phi_{2}^b + \mathcal{B}_3 ,
$$
wherein,
$$
\mathcal{A}_3 = -\mathcal{A}_2 + \dfrac{(\phi_{1}^{b} - \phi_{1}^{a})}{ H(1)}, \qquad  \mathcal{B}_3  = -\mathcal{B}_2 + \dfrac{(\phi_{1}^{b} - \phi_{1}^{a})}{ H(1)}.
$$
\end{rem}
\vskip .4in

\section{Permanent Charge and Channel Geometry Effects on Fluxes and I-V Relations.}\label{sec-num}  
\setcounter{equation}{0} 
In this section, we examine how permanent charges and the shape of the channel impact the movement of individual ions and the current-voltage (I-V) relations under electroneutrality conditions. When the absolute value of $Q_0$ (a measure of charge) is small, the flux ($\mathcal{J}_k$) for the $k$-th type of ion and the current ($\mathcal{I}$) can be expressed as follows:
$$
\mathcal{J}_k = D_kJ_{k0} + D_kJ_{k1} Q_0 + D_kJ_{k2} Q_0^2 + O(Q_0^3), \quad \mathcal{I} = \mathcal{I}_0
+ \mathcal{I}_1 Q_0 + \mathcal{I}_2 Q_0^2+ O(Q_0^3),
$$
where 
$$
\mathcal{I}_0 = z_1D_1J_{10} + z_2D_2J_{20}, \quad \mathcal{I}_1 = z_1D_1J_{11} + z_2D_2J_{21}, \quad \mathcal{I}_2 = z_1D_1J_{12} + z_2D_2J_{22}.
$$
The quantities $J_{1k}$ and $J_{2k}$, where $k=0,1,2$, capture the primary effects of permanent charges and channel shape on the flow of ions. We will analyze these quantities to understand their impact.\\


\subsection{Exploring First-Order Effects: Unraveling the Impact of Permanent Charges on Fluxes.}\label{sec-num01}

We start by simplifying specific findings from \cite{JLZ15} and presenting numerical results for the first-order terms. Initially, we articulate Theorem 4.8 from the same paper, offering numerical insights, and subsequently expand on our findings based on further numerical investigations.

\begin{thm}\label{thm-JLZ}
Suppose $B \neq 1$ where $B$ defined in \eqref{AB-LR}. Let $V_q^1$ and $V_q^2$ be as,
\begin{equation*}
\begin{aligned}
V_q^1 = V_q^1(L,R) = - \frac{\ln L - \ln R}{z_2(1-B)}, \qquad V_q^2 = V_q^2(L,R) = - \frac{\ln L - \ln R}{z_1(1-B)},
\end{aligned}
\end{equation*}
then the following cases arise:
\begin{itemize}
\item[\rm(i)] if $V_q^1 < 0 < V_q^2$, then, for $V > V_1^q$, a small positive $Q_0$ decreases $|J_1|$, and for $V < V_1^q$, a small positive $Q_0$ enhances $|J_1|$. Similarly, for $V > V_2^q$, a small positive $Q_0$ decreases $|J_2|$, and for $V < V_2^q$, a small positive $Q_0$ strengthens $|J_2|$; more precisely,
  \begin{itemize}
   \item[\rm(i1)] for $V \in (V_q^1,V_q^2)$,  $J_{10}J_{11}<0$ and $J_{20}J_{21}>0$;
   \item[\rm(i2)] for $V < V_q^1$,  $J_{10}J_{11}>0$ and $J_{20}J_{21}>0$;
     \item[\rm(i3)] for $V > V_q^2$,  $J_{10}J_{11}<0$ and $J_{20}J_{21}<0$;
  \end{itemize}
\item[\rm(ii)] if $V_q^1 > 0 > V_q^2$, then, for $V < V_q^1$, a small positive $Q_0$ decreases $|J_1|$, and for $V > V_q^1$, a small positive $Q_0$ enhances $|J_1|$. Similarly, for $V < V_q^2$, a small positive $Q_0$ decreases $|J_2|$, and for $V > V_q^2$, a small positive $Q_0$ strengthens $|J_2|$; more precisely,
  \begin{itemize}
    \item[\rm(ii1)] For $V \in (V_q^2, V_q^1)$, $J_{10}J_{11} < 0$ and $J_{20}J_{21} > 0$;
    \item[\rm(ii2)] For $V > V_q^1$, $J_{10}J_{11} > 0$ and $J_{20}J_{21} > 0$;
    \item[\rm(ii3)] For $V < V_q^2$, $J_{10}J_{11} < 0$ and $J_{20}J_{21} < 0$.
    \end{itemize} 
  
\end{itemize}

\end{thm}


The roots $V_q^1$ and $V_q^2$ in Theorem \ref{thm-JLZ} represent the solutions for $J_{10}J_{11}$ and $J_{20}J_{21}$, respectively, allowing us to investigate the impact of incorporating linear terms $J_{11}$ or $J_{21}$. However, this method becomes impractical for higher-order terms due to the complexity of computations, rendering analytical solutions unattainable.

The intricate nature of the second-order terms, specifically the fluxes $J_{12}$ and $J_{22}$ discussed in Section \ref{sec-second}, necessitates numerical approaches to determine their roots. Therefore, we turn to Python, leveraging the Numpy and Matplotlib libraries, to perform calculations for zeroth, first, and second-order terms. Additionally, numerical tools are employed to identify flux roots, facilitating the study of their signs across diverse regions.

We initially validate our computational approach and analytical results from Theorem \ref{thm-JLZ} by obtaining zeroth and first-order terms. Our numerical investigations not only confirm these results but also provide additional insights. Subsequently, we explore second-order terms and their impact on fluxes.

The numerical findings in Figures \ref{Fig-Heatmap_J-1} and \ref{Fig-Heatmap_J-2} affirm discussed scenarios, considering a fixed right boundary concentration $R=1$ while varying $L$ between $0$ and $2$. To validate, we employed two approaches: calculating $V_q^1$ and $V_q^2$ and determining signs on each interval, and numerically identifying roots without explicitly calculating $V_q^1$ and $V_q^2$. This alternate approach proves advantageous when incorporating second-order terms in the next section, where obtaining roots analytically could pose challenges.

In Figure \ref{Fig-Heatmap_J-1}, we initially present individual heatmaps indicating the signs of $J_{10}\cdot J_{11}$ and $J_{20}\cdot J_{21}$ to clarify how each flux change. The red region are the ones where $J_{10}$ and $J_{11}$ (or equivalently $J_{20}$ and $J_{21}$ on the top right panel) have the same signs and the blue ones are where the signs are opposite. 
One can cross-verify and compare the numerical findings with those presented in Theorem \ref{thm-JLZ}. Additionally, the numerical investigations showcased in the following figures reveal more intriguing outcomes.

Following that, the bottom plot of Figure \ref{Fig-Heatmap_J-1} highlights areas of overlap where both $J_{10}\cdot J_{11}$ and $J_{20}\cdot J_{21}$ share the same sign. The color scheme can be interpreted as follows: The red regions indicate areas where (small) positive $Q_0$ strengthens both $|J_1|$ and $|J_2|$, the blue regions denote areas where (small) positive $Q_0$ reduces both $|J_1|$ and $|J_2|$, and the purple regions represent areas where (small) positive $Q_0$ strengthens one of $|J_1|$ or $|J_2|$ while reducing the other.


	\begin{figure}[htbp]
 	\centering
 	\begin{subfigure}[b]{.87\textwidth}
    \includegraphics[width=\textwidth, height=2.in]{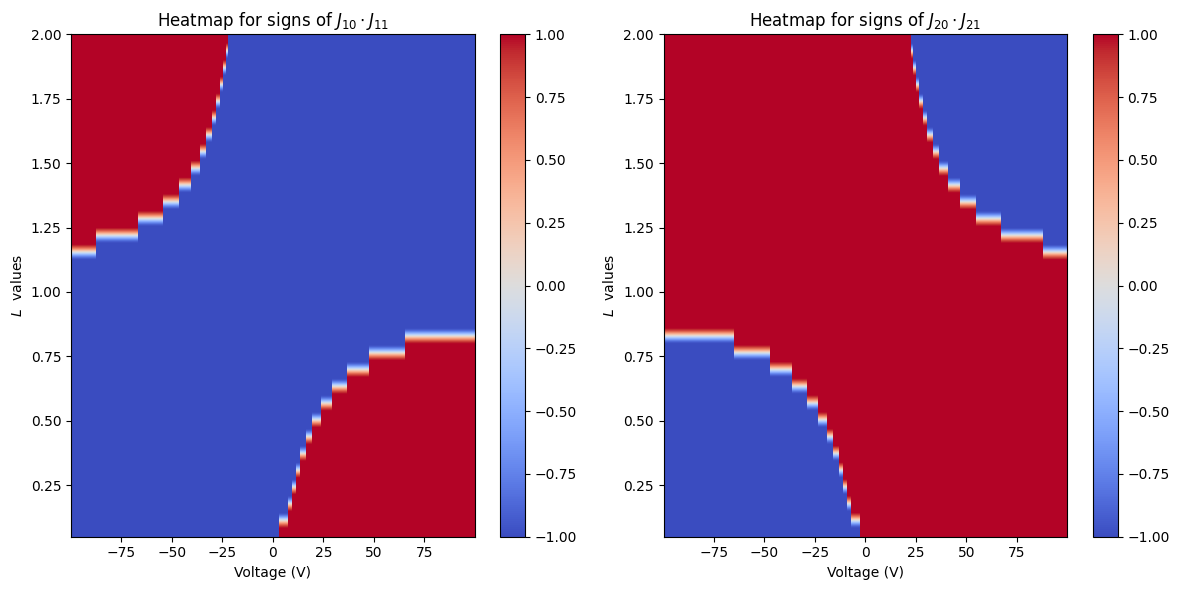}
\end{subfigure}
  	\begin{subfigure}{.46\textwidth}
 		\includegraphics[width=\textwidth, height=2.2in]{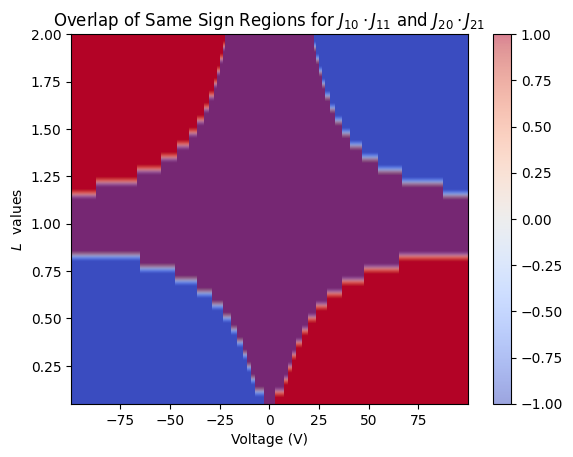}
 	\end{subfigure}	
 	\caption{\em  Visualization of heatmaps and overlapped regions indicating the sign agreement for the products $J_{10}\cdot J_{11}$ and $J_{20}\cdot J_{21}$. The concentration $L$ varies from zero to two while $R$ is fixed at $1$, shedding light on the impact of linear terms.
 	 }
 	 \label{Fig-Heatmap_J-1}
 \end{figure}
 

	\begin{figure}[h]
 	\centering
 	\begin{subfigure}[b]{.87\textwidth}
    \includegraphics[width=\textwidth, height=2.in]{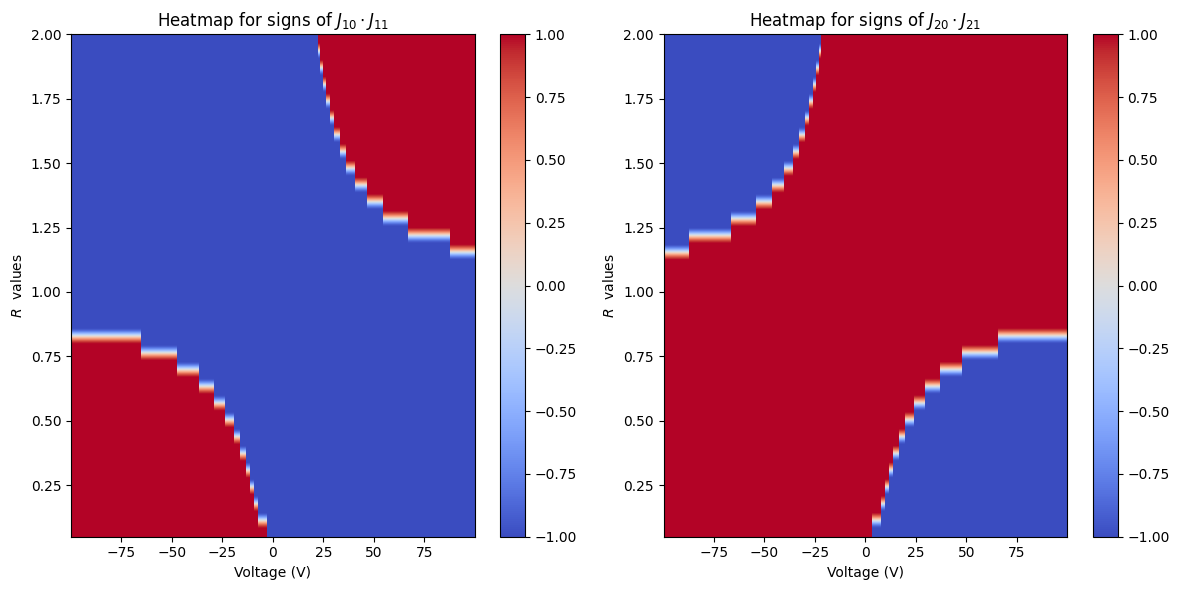}
\end{subfigure}
  	\begin{subfigure}{.46\textwidth}
 		\includegraphics[width=\textwidth, height=2.2in]{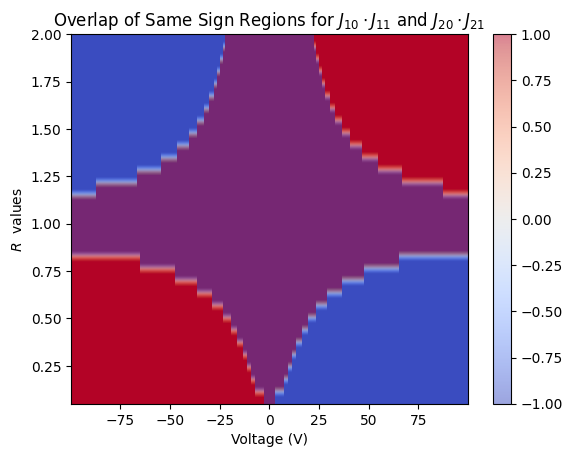}
 	\end{subfigure}	
 	\caption{\em  Heatmaps and overlaps of regions with same sign for $J_{10}\cdot J_{11}$ and  $J_{20}\cdot J_{21}$, with varying concentration $R$ from zero to two and fixed concentration $L=1$, to study the role of linear terms.
 	 }
 	 \label{Fig-Heatmap_J-2}
 \end{figure}
The theoretical analysis of complex second-order terms in equations provided in Proposition \ref{prop-lastJ} is challenging. As a result, we use computational methods to explore how permanent charges affect ion movement and the membrane's electrical behavior, focusing on the current-voltage (I-V) relation.
We analyze and compare these outcomes to scenarios without permanent charges, examining how these differences affect membrane performance. Then we study higher order contributions of permanent charges. Our numerical investigation delves into understanding the intricate interactions of permanent charges, shedding light on their influence on crucial electrical properties. Through this exploration, our aim is to advance our comprehension of the system's behavior and offer valuable insights to the academic community.

\subsection{Beyond the Basics: Investigating Second-Order Permanent Charge Effects on I-V Relations.}\label{sec-num2}
In this section, we explore the implications of introducing the $Q^2$ term into the expressions.  Additional comprehensive and noteworthy findings have been uncovered. Utilizing heatmaps to analyze the signs of $J_{10} \cdot J_{11} \cdot J_{12}$ and $J_{20} \cdot J_{21} \cdot J_{22}$ has unveiled more subtle insights, particularly emphasizing the distinct role played by the $Q^2$ terms in shaping the outcomes. Here, however, through Figure \ref{Fig-Q2-Linear-Quad} we first illustrate the transformative effects of incorporating the $Q^2$ term, shifting the behavior from linear to quadratic. We then would like to highlight that (small) positive $Q_0$ may reduce or strengthen $|J_1|$ and $|J_2|$ through second-order terms $Q_0^2$.

	\begin{figure}[htbp]
 	\centering
 	\begin{subfigure}{.8\textwidth}
 		\includegraphics[width=\textwidth]{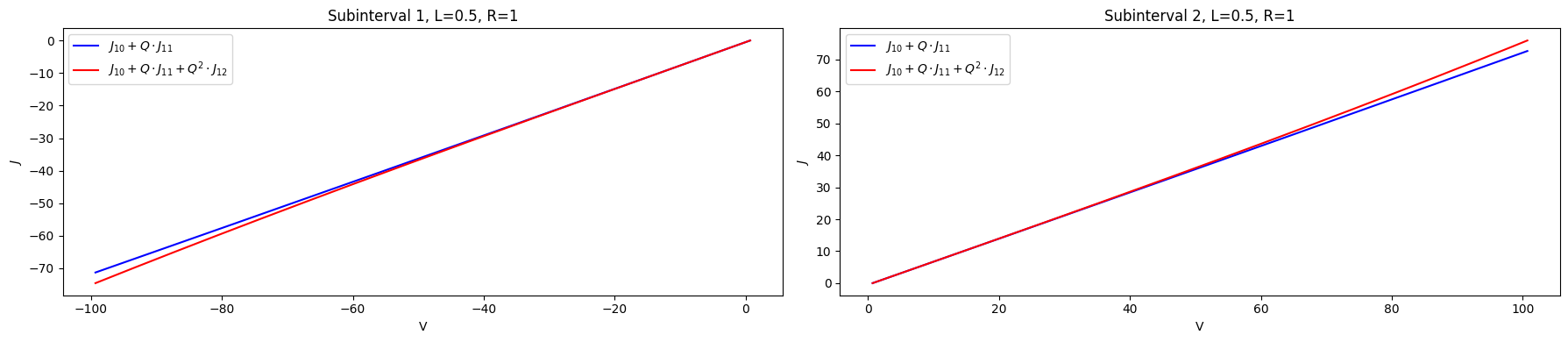}
 	\end{subfigure}
 	\begin{subfigure}{.8\textwidth}
 		\includegraphics[width=\textwidth]{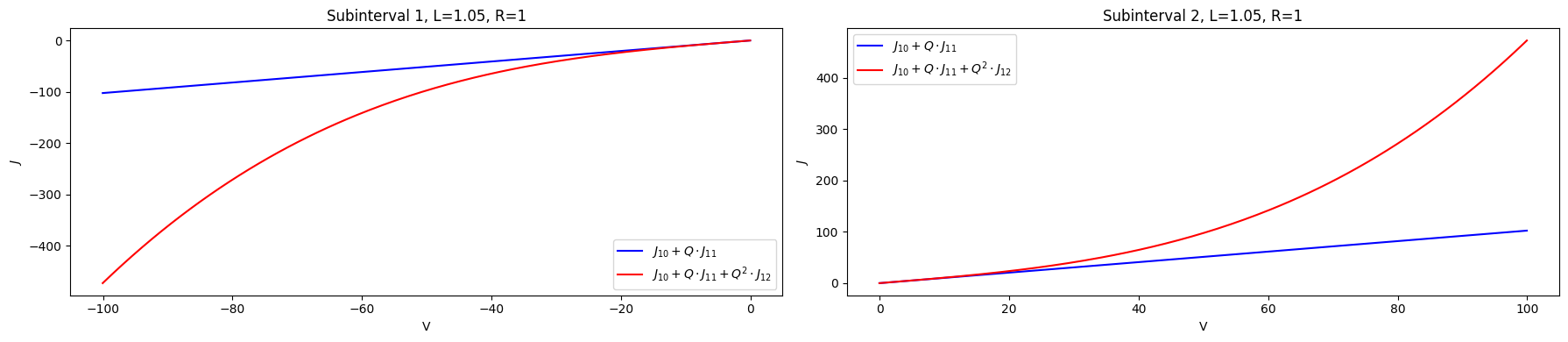}
 	\end{subfigure}
 	\caption{\em  Linear and quadratic approximations of flux $J_1$ for boundary concentrations $L=0.5, R=1$ (distant) and $L=1.05, R=1$ (close) with enhanced $|J_1|$ due to $Q_0^2$ terms.
 	 }
 	 \label{Fig-Q2-Linear-Quad}
 \end{figure}
 

Figure \ref{Fig-Q2-Linear-Quad} has limitations as it represents a specific case, making it non-representative of other scenarios. Even with similar figures, it struggles to clearly convey whether the quadratic term diminishes or enhances the flux. To address this, Figure \ref{Fig-Q2-Jprod-L0.97-1.07R1} offers more informative outcomes by depicting the products of $(J_{10} + Q \cdot J_{11})\cdot J_{12}$ and $(J_{20} + Q \cdot J_{21})\cdot J_{22}$.

 
\begin{figure}[htbp]
 	\centering
 	\begin{subfigure}{.8\textwidth}
 		\includegraphics[width=\textwidth]{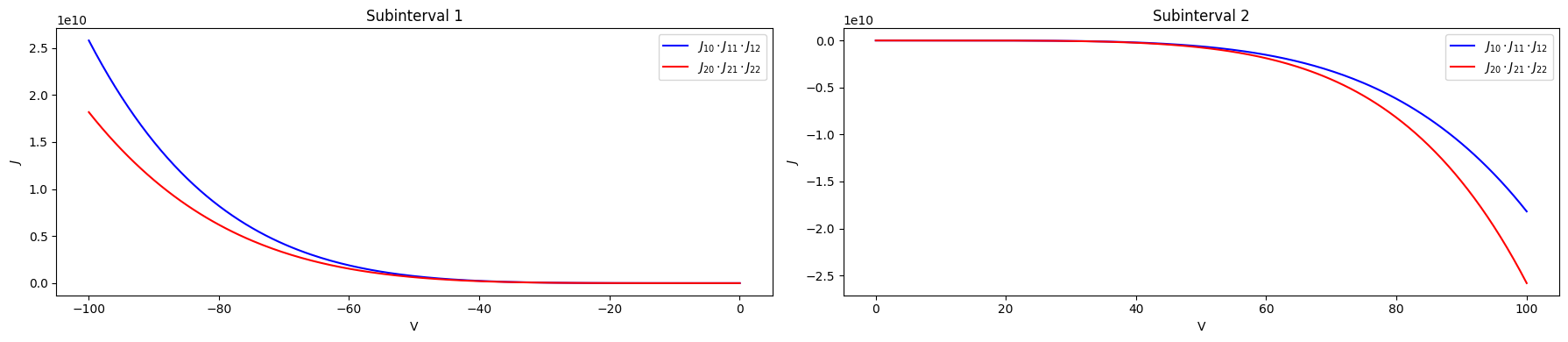}
 	\end{subfigure}
 	\begin{subfigure}{.8\textwidth}
 		\includegraphics[width=\textwidth]{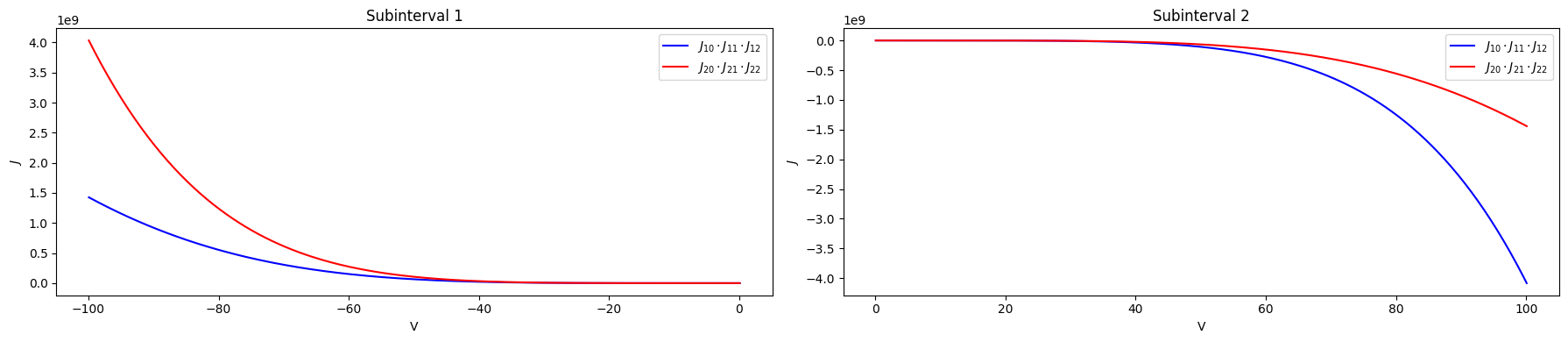}
 	\end{subfigure}
 	\caption{\em  Linear and quadratic approximations of flux $J_1$ for the boundary concentrations $L=0.5, R=1$ (distant) and $L=1.05, R=1$ (close) as a case where $Q_0^2$ terms strengthen $|J_1|$.
 	 }
 	 \label{Fig-Q2-Jprod-L0.97-1.07R1}
 \end{figure}
 
\section{Concluding Remarks and Future Work.}\label{sec-conclusion}
In this study, we presented a comprehensive exploration of ion channel dynamics, focusing on the intricate influence of permanent charges. Theoretical and numerical analyses have been combined to unveil the qualitative shifts in fluxes, boundary concentrations, and electric potentials at higher-order contributions of permanent charge. The investigation has delved into the subtle interplay between boundary conditions and channel geometry, elucidating the nuanced impact of permanent charges on ion channel behavior.
Our findings contribute to the understanding of ion electrodiffusion, shedding light on the complex interactions that arise due to permanent charges. The systematic perturbation analysis, spanning zeroth, first, and second-order solutions, has provided valuable insights into the behavior of the system under the influence of small permanent charges.
As we conclude this study, avenues for further research emerge.

Exploring Local Hard-Sphere PNP systems, which account for finite ion sizes, offers valuable insights into the dynamics of ionic channels by considering ion sizes $d$ \cite{LM22}. However, the computations become more complex in this case. A fascinating aspect of this study involves investigating higher-order solutions concerning ion size $d$ and permanent charge $Q_0$, specifically deriving $Q_0^2, Qd,$ and $d^2$ solutions. we derived solutions involving $Q_0^2$ in this manuscript. The work presented in \cite{FLMZ22} delves into the higher-order effects of ion size and provides $d^2$ solutions. Additionally, the paper \cite{LM22} examines PNP models with ion size and permanent charge, and to complete the puzzle, one must carefully derive $Qd$ terms from that paper. By assembling all these quadratic terms, a more accurate exploration of the higher-order impacts of ion size and permanent charge becomes possible.

Additionally, the application of advanced numerical techniques and simulations may offer a more detailed understanding of ion channel behavior in complex biological environments.
Further investigations could also delve into the impact of permanent charges on specific ion channel types, allowing for a more targeted analysis of their behavior. Moreover, experimental validation and comparison with existing biological data would provide a bridge between theoretical insights and real-world observations, enhancing the practical relevance of our findings.

\section{Acknowledgment. } 
The higher-order solutions presented in this manuscript smoothly extend from the foundational lower-order results, employing a consistent approach and methodology. The solutions maintain regularity concerning the permanent charge. To enhance clarity, certain intricate calculations, extensive in nature, have been intentionally omitted from the manuscript. Readers or reviewers with inquiries about specific sections of this manuscript are welcome to contact the authors for further clarification. The authors are readily available to provide detailed explanations and assistance as needed.



\begin{thebibliography}{100} \small


 \bibitem{Bar} V. Barcilon,
Ion flow through narrow membrane channels: Part I.
{\em SIAM J. Appl. Math.} {\bf 52} (1992), 1391-1404.


 \bibitem{Bie11} P. M. Biesheuvel,
Two-fluid model for the simultaneous flow of colloids and fluids in porous
media.
{\em J. Colloid Interface Sci.,} {\bf 355} (2011), 389-395.


\bibitem{CEJS95} D. Chen, R. Eisenberg, J. Jerome, and C. Shu, Hydrodynamic model of temperature change
in open ionic channels.
{\em Biophys. J.,} {\bf 69} (1995), 2304–2322.




\bibitem{Eis} B. Eisenberg,
Ion Channels as Devices.
{\em Journal of Computational Electronics} {\bf 2} (2003), 245-249.


\bibitem{Eis1} B. Eisenberg,
Proteins, Channels, and Crowded Ions.
{\em Biophysical Chemistry} {\bf 100} (2003),  507 - 517.


 \bibitem{Eis4} R. S. Eisenberg,
 Atomic Biology, Electrostatics and Ionic Channels.
 {\em  In New Developments and Theoretical Studies of Proteins,
 R. Elber, Editor.} 1996, World Scientific: Philadelphia.  269-357.



\bibitem{EL07} B. Eisenberg and W. Liu,
Poisson-Nernst-Planck systems for ion channels  with  permanent charges.
{\em SIAM J. Math. Anal.} {\bf 38} (2007), 1932-1966.

\bibitem{ELM19}
B. Eisenberg, W. Liu, and H. Mofidi, Effects of diffusion coefficients on reversal potentials in ionic channels.
{\em preprint arXiv:2311.02895.}, 2019.


\bibitem{FLMZ22}  Y. Fu, W. Liu, H. Mofidi and M. Zhang, Finite Ion Size Effects on Ionic Flows via Poisson-Nernst-Planck Systems: Higher Order Contributions. 
{\em Journal of  Dynamics and Differential Equations}  (2022),   1-25.


\bibitem {Gil99} D. Gillespie,
A singular perturbation analysis of the Poisson-Nernst-Planck system: Applications to Ionic Channels.
{\em Ph.D Dissertation}, Rush University at Chicago, 1999.


\bibitem{GO68} R. J. Gross and J. F. Osterle, Membrane transport characteristics of ultra fine capillary.
 {\em J.
Chem. Phys.,} {\bf 49} (1968),   228-234.


\bibitem{HH52} A. L. Hodgkin and A. F. Huxley, A quantitative description of membrane current and its application to conduction and excitation in nerve.
 {\em J. Physiol. (Lond.)} {\bf 117} (1952),   500-544.
 
 
\bibitem{HK55} A. L. Hodgkin and R. D. Keynes, The potassium permeability of a giant nerve fibre.
 {\em J. Physiol.} {\bf 128} (1955),   61-88.
 
 





\bibitem{HEL10} Y. Hyon, B. Eisenberg,  and C. Liu,
A mathematical model for the hard sphere repulsion in ionic solutions.
 {\em Commun. Math. Sci.} {\bf 9} (2010),   459-475.
 

\bibitem{HLE12} Y. Hyon, C. Liu, and B. Eisenberg, PNP equations with steric effects: A model of ion flow
through channels,
 {\em J. Phys. Chem.} {\bf 116} (2012),   11422–11441.
 

\bibitem{IR02} W. Im and B. Roux, Ion permeation and selectivity of OmpF porin: A theoretical study based
on molecular dynamics, Brownian dynamics, and continuum electrodiffusion theory,
 {\em J. Mol. Biol.} {\bf 322} (2002),   851-869.


 \bibitem{JL12}   S. Ji and W. Liu,
Poisson-Nernst-Planck systems for ion flow with density functional
theory for  hard-sphere potential: I-V relations and critical potentials. Part I: Analysis.
{\em J. Dynam. Differential Equations} {\bf 24} (2012), 955-983.


\bibitem{JLZ15} S. Ji, W. Liu, and M. Zhang,
Effects of (small) permanent charge and channel geometry on ionic flows
via classical Poisson-Nernst-Planck models,
{\em SIAM J. Appl. Math.} {\bf 75} (2015), 114-135.



\bibitem{JEL19} S. Ji, B. Eisenberg, and W. Liu, Flux Ratios and Channel Structures,
{\em J. Dyn. Diff. Equat.} {\bf 31} (2019), 1141-1183.



\bibitem {Liu05} W. Liu,
Geometric singular perturbation approach to steady-state Poisson–Nernst–Planck systems.
{\em  SIAM J. Appl. Math.,} {\bf 65} (2005),   754-766.



\bibitem {Liu09} W. Liu,
One-dimensional steady-state Poisson-Nernst-Planck systems for ion channels with multiple ion species.
{\em  J. Differential Equations} {\bf 246} (2009),   428-451.


\bibitem {Liu18} W. Liu,
One-dimensional A flux ratio and a universal property of permanent charges effects on
fluxes.
{\em  Comput. Math. Biophys.} {\bf 6} (2018),   28-40.





\bibitem {LM22} W. Liu  and H. Mofidi,
Local Hard-Sphere Poisson-Nernst-Planck Models for Ionic Channels with Permanent Charges.
{\em  preprint arXiv:2203.09113}, 2022.



\bibitem {LTZ12} W. Liu, X. Tu, and M. Zhang, Poisson-Nernst-Planck systems for ion flow with density
functional theory for hard-sphere potential: I-V relations and critical potentials. Part II:Numerics. {\em J. Dynam. Differential Equations}{\bf 24} (2012),  985-1004.



\bibitem {LW10} W. Liu  and B. Wang,
Poisson-Nernst-Planck systems for narrow tubular-like membrane channels.
{\em J. Dynam. Differential Equations} {\bf 22} (2010),  413-437.

\bibitem {MGNHEB09} A. Malasics, D. Gillespie, W. Nonner, D. Henderson, B. Eisenberg, and D. Boda,
Protein structure and ionic selectivity in calcium channels: Selectivity filter size, not shape, matters, Biochim.
{\em J. Biophys. Acta,} {\bf 1788} (2009),  2471–2480.





\bibitem{M20} H. Mofidi, Geometric Mean of Concentrations and Reversal Permanent Charge in Zero-Current Ionic Flows via Poisson-Nernst-Planck Models.
{\em preprint arXiv:2009.09564.}, 2020.


\bibitem{M21} H. Mofidi, Reversal permanent charge and concentrations in ionic flows via Poisson-Nernst-Planck models.
{\em Quart. Appl. Math.} {\bf 79} (2021),  581-600.


\bibitem{M23}
H. Mofidi, Bifurcation of Flux Ratio in Ionic Flows via a PNP model.
{\em preprint arXiv:2311.02895.}, 2023.




\bibitem{ML20} H. Mofidi and W. Liu, 
Reversal potential  and reversal permanent charge with unequal diffusion coefficients via classical Poisson--Nernst--Planck models.
{\em  SIAM J. Appl. Math. } {\bf 80} (2020), 1908-1935.




\bibitem{MEL20} H. Mofidi, B. Eisenberg and W. Liu,
Effects of Diffusion Coefficients and Permanent Charge on Reversal Potentials in Ionic Channels.
{\em  Entropy} {\bf 22} (2020), 325(1-23).





\bibitem {NE98} W. Nonner  and R. S.  Eisenberg,
Ion permeation and glutamate residues linked by Poisson-Nernst-Planck theory in L-type Calcium channels.
{\em Biophys. J.} {\bf 75} (1998),  1287-1305.




\bibitem {RABI04} B. Roux, T. W. Allen, S. Berneche, and W. Im, Theoretical and computational models of
biological ion channels.
{\em Quart. Rev. Biophys.} {\bf 37} (2004),  15-103.



\bibitem {SR81} V. Sasidhar and E. Ruckenstein, Electrolyte osmosis through capillaries.
{\em J. Colloid Interface
Sci.} {\bf 82} (1981),  439-1457.



\bibitem{SNE} Z. Schuss, B. Nadler,  and R. S.  Eisenberg.
Derivation of Poisson and Nernst-Planck equations in a bath and
channel from a molecular model.
{\em Physical Review E(3)} {\bf 64} (2001),  1-14.


\end{thebibliography}
\end{document}